\documentclass[smallextended,runningheads]{svjour3}
\usepackage{amsmath,amssymb,mathtools,mathabx,rotating}
\journalname{Submitted article}
\newcommand*{\id}{{\mathrm{id}}}
\newcommand*{\Ab}{{\mathbb A}}
\newcommand*{\Kb}{{\mathbb K}}

\newcommand*{\la}{{\langle}}
\newcommand*{\ra}{{\rangle}}
\newcommand*{\sh}
{{\,\begin{sideways}\begin{sideways}\begin{sideways}
$\exists$\end{sideways}\end{sideways}\end{sideways}}}
\newcommand*{\lb}{{[}}
\newcommand*{\rb}{{]}}

\newcommand*{\rH}{{\mathrm H}}

\newcommand*{\R}{{\mathbb R}}
\newcommand*{\CC}{{\mathbb{C}}}

\newcommand*{\Sb}{{\mathbb S}}
\newcommand*{\Nb}{{\mathbb N}}

\newcommand*{\quas}{{\ast}}

\newcommand*{\pa}{{\partial}}
\newcommand*{\cL}{{\mathrm L}}

\newcommand*{\rd}{{\mathrm d}}
\newcommand*{\End}{{\mathrm{End}}}

\newcommand*{\strat}{\,{\stackrel{\text{\tiny \textsf{o}}}{\phantom{\_}}}\,\,}

\smartqed

\begin{document}

\title{The exponential Lie series for continuous semimartingales}
\author{Kurusch Ebrahimi--Fard \and Simon~J.A.~Malham \and Frederic Patras \and Anke Wiese}
\titlerunning{Exponential Lie series}
\authorrunning{Ebrahimi--Fard, Malham, Patras and Wiese}

\institute{Kurusch Ebrahimi--Fard \at
Instituto de Ciencias Matem\'aticas,
Consejo Superior de Investigaciones Cient\'{i}ficas,
C/ Nicol\'as Cabrera, no. 13-15, 28049 Madrid, Spain
\and
Simon J.A. Malham \and Anke Wiese \at
Maxwell Institute for Mathematical Sciences,
and School of Mathematical and Computer Sciences,
Heriot-Watt University, Edinburgh EH14 4AS, UK
\and
Frederic Patras \at
Laboratoire J.A. Dieudonn\'e,
UMR CNRS-UNS No. 7351,
Universit\'e de Nice Sophia-Antipolis,
Parc Valrose,
06108 NICE Cedex 2, France
}

\date{25th June 2015}
\voffset=10ex

\maketitle
\begin{abstract}
We consider stochastic differential systems 
driven by continuous semimartingales and governed
by non-commuting vector fields. We prove that
the logarithm of the flowmap is an exponential 
Lie series. This relies on a natural change of basis
to vector fields for the associated quadratic covariation
processes, analogous to Stratonovich corrections. 
The flowmap can then be expanded as a 
series in compositional powers of vector fields 
and the logaritm of the flowmap can thus be expanded
in the Lie algebra of vector fields. Further, we give
a direct self-contained proof of the corresponding 
Chen--Strichartz formula which provides an explicit
formula for the Lie series coefficients. Such 
exponential Lie series are important in the development
of strong Lie group integration schemes that ensure
approximate solutions themselves lie in any homogeneous
manifold on which the solution evolves.
%\PACS{...}                                                
%\CRclass{...}
\end{abstract}

\section{Introduction}\label{sec:intro}
We are concerned with It\^o stochastic differential systems 
driven by continuous semimartingales and governed
by non-commuting vector fields of the following form
\begin{equation*}
Y_t=Y_0+\sum_{i=1}^d\int_0^t V_i(Y_\tau)\,\rd X_\tau^i,
\end{equation*}
for time $t\in[0,T]$ for some $T>0$.
Here the solution process $Y_t$ is $\R^N$-valued for some $N\in\Nb$.
For each $i=1,\ldots,d$, the  $X_t^i$ are driving scalar 
continuous semimartingales, and associated with each
are governing vector fields $V_i$ which are
sufficiently smooth and in general non-commuting. 
Our goal herein is to compute the logarithm of the flowmap for such a system, 
i.e.\/ the exponential series for the flowmap, and establish that it is a Lie series.
The exponential series for the flowmap for stochastic differential systems 
driven by general continuous semimartingales was derived in Ebrahimi-Fard,
Malham, Patras and Wiese~\cite{E-FMPW}. What we achieve that is new in 
this paper is we:
\begin{enumerate}
\item Establish the abstract algebraic structures that underly the flowmap
and computation of functions of the flowmap in the context of general continuous semimartingales;   
\item Show by a suitable change of coordinates, the exponential series is a Lie series;
\item Give a direct self-contained proof of the corresponding Chen--Strichartz formula 
which provides an explicit formula for the Lie series coefficients.
\end{enumerate}
The key idea that underlies establishing the exponential series as a Lie series
is to express the flowmap in Fisk--Stratonovich form for which the standard rules 
of calculus apply; see Protter~\cite{Protter}. 
A crucial integral ingredient in this step is that the 
Fisk--Stratonovich formulation of the flowmap can be expanded in
a basis of terms involving solely compositions of vector fields---without
any second order partial differential operators. We can then compute the logarithm 
of the Fisk--Stratonovich representation of the flowmap.  This can be accomplished 
in principle via the classical Chen--Strichartz formula using the shuffle relations satisfied 
by multiple Fisk--Stratonovich integrals as well as utilizing the Dynkin--Friedrichs--Specht--Wever
Theorem to expand the logarithm in Lie polynomials of the vector fields.
We subsequently convert the multiple Fisk--Stratonovich integrals back into multiple 
It\^o integrals. This procedure thus generates an It\^o exponential Lie series. 
That the logarithm of the flowmap is in fact an exponential Lie series 
is important for example, for the development of strong stochastic Lie group 
integration methods. See Malham and Wiese~\cite{MW} for the development
of such methods for Stratonovich stochastic differential systems
driven by Wiener processes, for example those based on the Castell--Gaines
numerical simulation approach, see Castell and Gaines~\cite{CG,CG2}.

The development of exponential solution series for deterministic systems 
originates with the work of Magnus~\cite{Magnus} and Chen~\cite{Chen} 
in the 1950's, and more recently with Strichartz~\cite{S}.
Its development and early application to stochastic systems is represented
by the work of Azencott~\cite{Azencott}, Ben Arous~\cite{BA},
Castell and Gaines~\cite{CG,CG2} and Baudoin~\cite{Baudoin}. 
Also see Fliess~\cite{Fliess} and Lyons~\cite{L} 
for its development in control and theory of rough paths, respectively.
The shuffle product was cemented in firm foundations by the 
work of Eilenberg and Mac Lane~\cite{EM} and Sch\"utzenberger~\cite{Schutzenberger},
also in the 1950's. The quasi-shuffle product is a natural 
extension of the shuffle product, for example to multiple It\^o
integrals. For a selective insight into its recent development 
in this context, see Gaines~\cite{G,G2}, Hoffman~\cite{H}, 
Ebrahimi--Fard and Guo~\cite{E-FG}, Hoffman and Ihara~\cite{HI} 
and Curry, Ebrahimi--Fard, Malham and Wiese~\cite{CE-FMW}. 

Our paper is structured as follows. 
In Section~\ref{sec:flowmap} we derive the It\^o chain rule and flowmap
for systems driven by continuous semimartingales. Then in Section~\ref{sec:algebra} 
we establish the abstract algebraic structures that underpin the flowmap and its logarithm.
We endeavour to keep the connection to the stochastic differential system 
of interest and provide illustrative examples. We define the transformation to
Stratonovich form we require in Section~\ref{sec:Lie} and prove that the exponential series 
is a Lie series. Our direct self-contained derivation of the Chen--Strichartz coefficients 
is provided in Section~\ref{sec:Chen--Strichartz}.
Lastly, we provide some concluding remarks in Section~\ref{sec:conclu}.

\section{It\^o chain rule and flowmap}\label{sec:flowmap}
Consider the It\^o stochastic differential system given in the introduction
of the form
\begin{equation*}
Y_t=Y_0+\sum_{i=1}^d\int_0^t V_i(Y_\tau)\,\rd X_\tau^i.
\end{equation*}
Here for each $i=1,\ldots,d$, the  $X_t^i$ are driving scalar 
continuous semimartingales on a filtered probability space 
$\bigl(\Omega,\mathcal F,(\mathcal F_t)_{t\geqslant0},P\bigr)$ satisfying
the \emph{usual conditions} of completeness and right-continuity.
We assume without loss of generality that the $X^i$ are chosen such
that the quadratic covariations $[X^i,X^j]=0$ for all $i\neq j$.
The $V_i$ are associated governing vector fields which we assume 
are sufficiently smooth and in general non-commuting. 
We suppose the solution process $Y_t$, 
which is $\R^N$-valued for some $N\in\Nb$, exists on some finite
or possibly infinite time interval.
In coordinates the vector fields $V_i$ for each $i=1,\ldots,d$ act as 
first order partial differential operators on any function 
$f\colon\R^N\to\R$ as follows
\begin{equation*}
V_i\colon f(Y)\mapsto\sum_{j=1}^NV_i^j(Y)\pa_{Y_j}f(Y).
\end{equation*}
For brevity we will often express this vector field action as
$(V_i\cdot\pa)f(Y)$ or $V_i\circ f\circ Y$.
\begin{definition}[Flowmap]
We define the flowmap $\varphi_t$ as the map prescribing the 
transport of the initial data $f\circ Y_0$ to the solution $f\circ Y_t$ 
at time $t$ for any smooth function $f$ on $\R^N$, 
i.e.\/ $\varphi_t\colon f\circ Y_0\mapsto f\circ Y_t$. 
\end{definition}
The solution $Y_t=\varphi_t\circ\id\circ Y_0$ corresponds to the choice 
$f=\id$, the identity map.
The It\^o chain rule is the key to developing the Taylor series
expansion for the solution $Y_t$ about the initial data. 
The It\^o chain rule implies that for any function $f\colon\R^N\to\R$, 
the quantity $f(Y_t)$ satisfies 
\begin{equation*}
f(Y_t)=f(Y_0)+\sum_{i=1}^d\!\int_0^t\!\bigl(V_i\cdot\pa\bigr)f(Y_\tau)\,\rd X_\tau^i
+\tfrac12\sum_{i=1}^d\!\int_0^t\!\bigl(V_i\otimes V_i\colon\pa^2\bigr)f(Y_\tau)\,\rd[X^i,X^i]_\tau,
\end{equation*}
see for example Protter~\cite{Protter}.
In this formula we have used the notation
\begin{equation*}
\bigl(V_i\otimes V_i\colon\pa^2\bigr)f(Y)\coloneqq\sum_{j,k=1}^NV_i^j(Y)V_i^k(Y)\pa_{Y_j}\pa_{Y_k}f(Y),
\end{equation*}
while for each $i=1,\ldots,d$ the terms $[X^i,X^i]$ represent the quadratic variation of $X^i$.
At this stage it makes sense to extend, first our set of driving continuous semimartingales
to include these quadratic variations, and second, our governing vector fields to include the
associated second order partial differential operators shown above. 
Thus for $i=1,\ldots,d$ we set $D_i\coloneqq V_i\cdot\pa$ and
\begin{equation*}
X^{[i,i]}\coloneqq[X^{i},X^{i}]
\qquad\text{and}\qquad
D_{[i,i]}\coloneqq\tfrac12 V_{i}\otimes V_{i}\colon\pa^2.
\end{equation*}
Then the It\^o chain rule takes the form 
\begin{equation*}
f\circ Y_t=f\circ Y_0+\sum_{a\in\Ab}\int_0^t D_a\circ f\circ Y_\tau\,\rd X_\tau^a,
\end{equation*}
where $\Ab$ denotes the alphabet set of letters $\{1,\ldots,d,[1,1],\ldots,[d,d]\}$.
Iterating this chain rule produces the formal Taylor series expansion for the solution
and thus flowmap given by
\begin{equation*}
\varphi_t=\sum_wI_w(t)D_w.
\end{equation*}
Here the sum is over all words/multi-indices $w$ that can be constructed from the alphabet
$\Ab$. All the stochastic information is encoded in the multiple
stochastic It\^o integrals $I_w=I_w(t)$ while the geometric information is 
encoded through the composition of partial differential operators $D_w$.
For a word $w=a_1\cdots a_n$ these terms are 
$D_w\coloneqq D_{a_1}\circ\cdots\circ D_{a_n}$ and
\begin{equation*}
I_w\coloneqq\int_{0\leqslant\tau_{2}\leqslant\cdots\leqslant\tau_{n}\leqslant t}
\,\rd X^{a_1}_{\tau_1}\cdots\,\rd X^{a_n}_{\tau_n}.
\end{equation*}
It is natural to abstract the solution flowmap
and view it as an object of the form
\begin{equation*}
\sum w\otimes w,
\end{equation*}
which lies in a tensor product of two algebras. The algebra on the
left is associated with multiple integrals and the algebra on the right is
associated with partial differential operators. The algebra
on the left should be endowed with a quasi-shuffle product, 
to reflect the fact that the real product
between two multiple It\^o integrals $I_u$ and $I_v$ generates a sum over
all multiple It\^o integrals generated by the quasi-shuffle of the words $u$
and $v$; we define this product precisely, presently. The algebra
on the right should be endowed with a concatenation product, to reflect
the fact that the composition of two differential operators $D_u$ and $D_v$
generates the differential operator equivalent to that represented by the 
concatenation of the words $u$ and $v$.
In the next section we define these underlying concatenation
and quasi-shuffle algebras, their corresponding Hopf algebras 
and the algebras associated with endomorphisms on them. 
These algebras prove useful in the following sections, they keep 
our proofs direct and succinct.

\section{Quasi-shuffle Hopf algebras and endomorphisms}\label{sec:algebra}
Our exposition here is based on 
Reutenauer~\cite{R}, Hoffman~\cite{H} and Hoffman and Ihara~\cite{HI}.
Let $\Ab$ denote a countable alphabet and $\Kb\Ab$ the vector space with
$\Ab$ as basis and $\Kb$ a field of characteristic zero. 
Suppose there is a commutative and associative product 
$[\,\cdot\,,\,\cdot\,]$ on $\Kb\Ab$. We use $\Kb\la\Ab\ra$ to denote 
the non-commutative polynomial %and formal power series algebra 
algebra over $\Kb$ generated by monomials
(or words) we can construct from the alphabet $\Ab$.
We denote by $\Ab^\ast$ the free monoid of words on $\Ab$. 
\begin{definition}[Bilinear form]
We define the bilinear form 
$\la\,\cdot\,,\,\cdot\,\ra\colon\Kb\la\Ab\ra\otimes\Kb\la\Ab\ra\to\Kb$ 
for any words $u,v\in\Ab^\ast$ to be
\begin{equation*}
\langle u,v\rangle
\coloneqq\begin{cases}
1, &~~\text{if}~u=v,\\
0, &~~\text{if}~u\neq v.
\end{cases}
\end{equation*}
\end{definition}
This is equivalent to the scalar product given 
in Reutenauer~\cite[p.~17]{R} and Hoffman~\cite[p.~57]{H}. 
For this scalar product, the free monoid $\Ab^\ast$ forms
an orthonormal basis. 
We will always assume that $\Ab$ equipped with $[\,\cdot\,,\,\cdot\,]$
satisfies the following finiteness condition: for all letters $c\in\Ab$
the cardinality of the set $\{a,b\in\Ab\colon\la[a,b],c\ra\neq0\}$ is finite.
It is satisfied, though not restricted to, when $\Ab$ is finite. The following
example illustrates the case of a possibly infinite alphabet.
\begin{example}                                                           
Consider a minimal family of general semimartingales given by $\{X^1,\ldots,X^d\}$
in the sense outlined in Curry, Ebrahimi--Fard, Malham and Wiese~\cite{CE-FMW}.
We do not restrict ourselves here to continuous semimartingales. 
However, a collection of independent continuous semimartingales,
or a collection of independent L\'evy processes, is a minimal family. 
We can construct a countable alphabet $\Ab$ as outlined therein as follows. 
With each semimartingale we associate a letter $1,\ldots,d$. 
In addition, inductively for $n\geqslant2$,
we assign a distinct new letter
for each nested quadratic covariation process                                 
$[X^{k_1},[X^{k_2},[\ldots[X^{k_{n-1}},X^{k_n}]\ldots]]]$     
with $k_i\in\{1,\ldots,d\}$ for $i=1,\ldots,n$, 
``provided it is not in the linear span of 
$\{X^1,\ldots,X^d\}$ and previously constructed ones''.
Due to commutativity the order of the letters in the
$n$-fold nested bracket is irrelevant and associativity means that we
can render all $n$-fold nested brackets to the canonical form of
left to right bracketing shown, or more conveniently $[X^{k_1},\ldots,X^{k_n}]$.
We denote the new distinct letters by $[k_1,\ldots,k_n]$. 
Hence our underlying countable alphabet $\Ab$
consists of the letters $1,\ldots,d$ and all possible nested brackets
$[\,\cdot\,,\,\cdot\,]$ generated in this manner. 
For convenience we set $[X^i]\equiv X^i$ and thus also $[i]\equiv i$
on $\Kb\la\Ab\ra$.
\end{example}
We use $\Kb\la\Ab\ra$ to also denote the concatenation
algebra of words with concatentation as product. If
$u$ and $v$ are words in $\Kb\la\Ab\ra$, then their 
concatenation is $uv\in\Kb\la\Ab\ra$.
\begin{definition}[Quasi-Shuffle product]
For words $u,v$ and letters $a,b$ the quasi-shuffle product $\quas$ 
on $\Kb\la\Ab\ra$ is generated recursively by the formulae: 
$u\,\quas\, 1=1\,\quas\, u=u$, where `$1$' represents the empty word, and
\begin{equation*}
ua\,\quas\, vb=(u\,\quas\, vb)\,a+(ua\,\quas\, v)\,b+(u\,\quas\, v)\,[a,b].
\end{equation*}
\end{definition}
Endowed with this product $\Kb\la\Ab\ra$ is a commutative and associative 
algebra called the quasi-shuffle algebra which we denote by $\Kb\la\Ab\ra_\quas$; 
see Hoffman~\cite{H}.
In the special case when the generator $[\,\cdot\,,\,\cdot\,]$ is 
identically zero on $\Kb\Ab$, it reverts to the 
shuffle algebra $\Kb\la\Ab\ra_\sh$ of words with shuffle $\sh\,$ as product,
where $ua\,\sh\,\, vb=(u\,\sh\,\, vb)\,a+(ua\,\sh\,\, v)\,b$.
\begin{example}
The quasi-shuffle of the words 
$12$ and $34$ is given by
$12\,\quas\,34=1234+3412+1342+3142+1324+3124
+1[2,3]4+[1,3]42+3[1,4]2+[1,3]24+13[2,4]+31[2,4]+[1,3][2,4]$.
\end{example}
\begin{example}
A minimal family of semimartingales generates a quasi-shuffle algebra. This 
is proved in Curry \textit{et al.}\/ \cite{CE-FMW}.
\end{example}
\begin{definition}[Deconcatenation and de-quasi-shuffle coproducts]
We define the deconcatenation coproduct 
$\Delta\colon\Kb\la\Ab\ra\to\Kb\la\Ab\ra\otimes\Kb\la\Ab\ra$
for any word $w\in\Kb\la\Ab\ra$ by
\begin{equation*}
\Delta(w)\coloneqq\sum_{u,v}\la uv,w\ra\,u\otimes v.
\end{equation*}
We also define the de-quasi-shuffle coproduct
$\Delta'\colon\Kb\la\Ab\ra\to\Kb\la\Ab\ra\otimes\Kb\la\Ab\ra$
for any word $w\in\Kb\la\Ab\ra$ by
\begin{equation*}
\Delta'(w)\coloneqq\sum_{u,v}\la u\,\quas\, v,w\ra\,u\otimes v.
\end{equation*}
\end{definition}
The finiteness condition on $\Ab$ ensures that $\Delta'$ is well defined.
Endowed with the concatenation product and de-quasi-shuffle coproduct
$\Kb\la\Ab\ra$ is a Hopf algebra which we also denote by $\Kb\la\Ab\ra$.
No confusion should arise from the context.
In addition, when endowed with the quasi-shuffle product
and deconcatenation coproduct $\Kb\la\Ab\ra$ is another
Hopf algebra which we denote by $\Kb\la\Ab\ra_\quas$. The 
antipode in both cases is the signed reversing endomorphism.
We denote by $\End(\Kb\la\Ab\ra_\quas)$ the $\Kb$-module of linear
endomorphisms of $\Kb\la\Ab\ra_\quas$. 
\begin{definition}[Convolution products]
Suppose $X$ and $Y$ are two linear endomorphisms on the Hopf quasi-shuffle
algebra $\Kb\la\Ab\ra_\quas$. 
We define their quasi-shuffle convolution product $X\,\quas\, Y$ by the formula
$X\,\quas\, Y\coloneqq\mathrm{quas}\circ(X\otimes Y)\circ\Delta$,
where `$\mathrm{quas}$' denotes the quasi-shuffle product on $\Kb\la\Ab\ra_\quas$.
\end{definition}
\begin{remark}
We use the same notation for the quasi-shuffle convolution product 
as for the underlying product, no confusion should arise from the context.
There is also a concatenation convolution product 
$\mathrm{conc}\circ(X\otimes Y)\circ\Delta'$ 
on $\End(\Kb\la\Ab\ra)$, where `$\mathrm{conc}$' denotes the concatenation. 
\end{remark}
In other words, since deconcantenation $\Delta$ splits any word $w$ 
into the sum of all two-partitions $u\otimes v$ with $u,v\in\Ab^\ast$, 
including when $u$ or $v$ are the empty word $1$, we see that
\begin{equation*}
\bigl(X\,\quas\, Y\bigr)(w)=\sum_{uv=w}X(u)\,\quas\, Y(v).
\end{equation*}
Now consider the following algebra which plays an essential role hereafter,
\begin{equation*}
\Kb\la\Ab\ra_\quas\overline{\otimes}\,\,\Kb\la\Ab\ra.
\end{equation*}
This is the complete tensor product 
of the Hopf quasi-shuffle algebra on the left 
and the Hopf concatenation algebra on the right;
see Reutenauer~\cite[p.~18, 29]{R}.
It is itself an associative Hopf algebra. 
The product of any two elements in this tensored Hopf algebra,
which extends linearly, is naturally given by 
\begin{equation*}
(u\otimes u')(v\otimes v')=(u\,\quas\, u')\otimes(vv').
\end{equation*}
\begin{remark}
In our context, the significance of this tensor algebra
is that it is the natural abstract setting for the flowmap. 
\end{remark}
Any endomorphism $X\in\End(\Kb\la\Ab\ra_\quas)$ 
can be completely described by the image 
in $\Kb\la\Ab\ra_\quas\overline{\otimes}\,\,\Kb\la\Ab\ra$
of the map
\begin{equation*}
X\mapsto\sum_{w\in\Ab^\ast}X(w)\otimes w,
\end{equation*}
see Reutenauer~\cite[p.~29]{R}. Note that the identity endomorphism `$\id$'
on $\Kb\la\Ab\ra_\quas$ maps onto $\sum w\otimes w$.
Indeed the embedding
$\End(\Kb\la\Ab\ra_\quas)\to\Kb\la\Ab\ra_\quas\overline{\otimes}\,\,\Kb\la\Ab\ra$
defined by this map is an algebra homomorphism for the quasi-shuffle 
convolution product. The unit endomorphism $\nu$ on the algebra 
$\End(\Kb\la\Ab\ra_\quas)$ sends non-empty words to zero and the
empty word to itself. This embedding provides a mechanism for 
representing functions of $\sum w\otimes w$ in $\End(\Kb\la\Ab\ra_\quas)$.
Before we demonstrate this, we need the following.
\begin{definition}[Augmented ideal projector]
We use $J$ to denote the augmented ideal projector. This is the linear
endomorphism on $\Kb\la\Ab\ra_\quas$ or $\Kb\la\Ab\ra$ that sends every 
non-empty word to itself and the empty word to zero. From the definition
of the unit endomorphism $\nu$ given above we see that $J\equiv\id-\nu$.
\end{definition}
We observe we can apply a power series function such as the logarithm function
to the element $\sum_ww\otimes w$ in $\Kb\la\Ab\ra_\quas\overline{\otimes}\,\,\Kb\la\Ab\ra$
as follows. If `$1$' represents the empty word, $|w|$ represents the length of the word $w$
and $c_k\coloneq(-1)^{k-1}\frac{1}{k}$ for all $k\in\mathbb N$, then 
by direct computation we find
\begin{align*}
\log\Biggl(\sum_{w\in\Ab^\ast}w\otimes w\Biggr)
&=\sum_{k\geqslant1}c_k\Biggl(\sum_{w\in\Ab^\ast}w\otimes w-1\otimes1\Biggr)^k\\
&=\sum_{k\geqslant1}c_k\Biggl(\sum_{w\in\Ab^\ast\backslash\{1\}}w\otimes w\Biggr)^k\\
&=\sum_{k\geqslant1}c_k\sum_{u_1,\ldots,u_k\in\Ab^\ast\backslash\{1\}}
(u_1\,\quas\,\cdots\,\quas\, u_k)\otimes(u_1\cdots u_k)\\
&=\sum_{w\in\Ab^\ast}\Biggl(\sum_{k=1}^{|w|}c_k
\sum_{u_1\cdots u_k=w}u_1\,\quas\,\cdots\,\quas\, u_k\Biggr)\otimes w\\
&=\sum_{w\in\Ab^\ast}\Biggl(\sum_{k\geqslant1}c_kJ^{\quas k}\Biggr)\circ w\otimes w,
\end{align*}
where $J^{\quas k}$ denotes the $k$th quasi-shuffle convolution power
of the augmented ideal projector $J$. We emphasize the elements in the
partition $u_1\cdots u_k=w$ in the sum in the penultimate line are all non-empty.
Note that words cannot be deconcatenated further than all the letters it contains,
and thus $J^\quas(w)$ is zero if $w$ has length less than $k$. 
We conclude the action of the logarithm function power series on $\sum_ww\otimes w$
can be represented by a corresponding power series endomorphism in 
$\End(\Kb\la\Ab\ra_\quas)$. This will prove useful 
in Sections~\ref{sec:Lie} and \ref{sec:Chen--Strichartz} so 
we summarize the result as follows.
\begin{lemma}[Logarithm convolution power series]\label{lemma:logpowerseries}
The logarithm of the element $\sum_ww\otimes w$ is given by
\begin{equation*}
\log\Biggl(\sum_{w\in\Ab^\ast}w\otimes w\Biggr)=
\sum_{w\in\Ab^\ast}\log^\quas(\id)\circ w\otimes w,
\end{equation*}
where
\begin{equation*}
\log^\quas(\id)\coloneqq\sum_{k\geqslant1}\frac{(-1)^{k-1}}{k}J^{\quas k}.
\end{equation*}
We will often abbreviate $\log^\quas(\id)\circ w$ to $\log^\quas(w)$.
\end{lemma}
We also note that equivalently, the embedding 
$\End(\Kb\la\Ab\ra)\to\Kb\la\Ab\ra_\quas\overline{\otimes}\,\,\Kb\la\Ab\ra$
given by 
\begin{equation*}
Y\mapsto\sum_{w\in\Ab^\ast}w\otimes Y(w),
\end{equation*}
is an algebra homomorphism for the concantenation convolution product.
\begin{definition}[Adjoint endomorphisms]\label{def:adjoints1}
Two endomorphisms $X$ and $Y$ are adjoints if the images of
\begin{equation*}
X\mapsto \sum_{w}w\otimes X(w)
\qquad\text{and}\qquad
Y\mapsto \sum_{w}Y(w)\otimes w,
\end{equation*}
are equal. Hereafter we use $X^\dag$ to denote the adjoint of $X$.
\end{definition}
This coincides with $X$ and $X^\dag$ 
being adjoints in the sense 
$\bigl\langle X^{\dag}(u),v\bigr\rangle
=\bigl\langle u,X(v)\bigr\rangle$ 
for all $u,v\in\Ab^\ast$; see Reutenauer~\cite[Section~1.5]{R}.

There is a natural isomorphism between the Hopf shuffle and quasi-shuffle
algebras discovered by Hoffmann~\cite{H} which will play an important
role here; also see Foissy, Patras and Thibon~\cite{FPT} for a theoretical perspective.
To describe the isomorphism succinctly we need to introduce
the notion of composition action on words. For any natural number $n$, we use 
$\CC(n)$ to denote the set of compositions
of $n$. A given composition $\lambda$ in $\CC(n)$ will have say $\ell\leqslant n$
components so $\lambda=(\lambda_1,\ldots,\lambda_\ell)$. For such a
$\lambda$ we define the following simple multi-index functions, $|\lambda|\coloneqq \ell$
as well as 
\begin{equation*}
\Sigma(\lambda)\coloneqq \lambda_1+\cdots+\lambda_\ell,\quad
\Pi(\lambda)\coloneqq \lambda_1\cdots\lambda_\ell
\quad\text{and}\quad
\Gamma(\lambda)\coloneqq \lambda_1!\cdots\lambda_\ell!.
\end{equation*}
We now define the composition 
action on words and the exponential map from Hoffman~\cite{H}.
\begin{definition}[Composition action]
For a given word $w=a_1\cdots a_n$ and composition 
$\lambda=(\lambda_1,\ldots,\lambda_\ell)$ in $\CC(n)$
we define the action $\lambda\circ w$ to be
\begin{equation*}
\lambda\circ w\coloneqq 
\lb a_1\cdots a_{\lambda_1}\rb
\lb a_{\lambda_1+1}\cdots a_{\lambda_1+\lambda_2}\rb\cdots
\lb a_{\lambda_1+\cdots+\lambda_{\ell-1}+1}\cdots a_{n}\rb,
\end{equation*}
where the brackets are concatenated.
Here for any word $w=a_1\cdots a_n$ the notation $\lb w\rb$ denotes
the $n$-fold nested bracket described above. 
\end{definition}
\begin{definition}[Hoffman exponential]
We define the map $\exp_\rH\colon\Kb\la\Ab\ra_\sh\to\Kb\la\Ab\ra_\quas$
as that which leaves the empty word unchanged and for any non-empty word $w$ is
\begin{equation*}
\exp_\rH(w)\coloneqq\sum_{\lambda\in\CC(|w|)}\frac{1}{\Gamma(\lambda)}\lambda\circ w.
\end{equation*}
Its inverse 
$\log_\rH\colon\Kb\la\Ab\ra_\quas\to\Kb\la\Ab\ra_\sh$
for any word $w$ is given by 
\begin{equation*}
\log_\rH(w)\coloneqq\sum_{\lambda\in\CC(|w|)}
\frac{(-1)^{\Sigma(\lambda)-|\lambda|}}{\Pi(\lambda)}\lambda\circ w.
\end{equation*}
\end{definition}
Hoffman~\cite{H} proved the exponential map is an isomorphism 
from the Hopf shuffle algebra $\Kb\la\Ab\ra_{\sh}$ 
to the Hopf quasi-shuffle algebra $\Kb\la\Ab\ra_{\quas}$.
\begin{example}
In the case of the word $w=a_1a_2a_3$ the Hoffman exponential is given by
$\exp_{\rH}(a_1a_2a_3)=a_1a_2a_3+\frac12[a_1,a_2]a_3+\frac12a_1[a_2,a_3]+\frac16[a_1,a_2,a_3]$.
\end{example}
%
% See Hoffman paper Theorem 4.2...
%
The adjoint of the Hoffman exponential 
$\exp_{\rH}^\dag\colon\Kb\la\Ab\ra_{\mathrm{conc},\Delta'}\to\Kb\la\Ab\ra_{\mathrm{conc},\delta}$
is an isomorphism and defined explicitly as follows. Here we distinguish between two Hopf
algebras with concatenation as product, we have $\Kb\la\Ab\ra_{\mathrm{conc},\Delta'}$
with de-quasi-shuffle $\Delta'$ as coproduct and $\Kb\la\Ab\ra_{\mathrm{conc},\delta}$ 
with deshuffle $\delta$ as coproduct. 
Then for any letter $a$ from the alphabet $\Ab$ we have 
\begin{equation*}
\exp_{\rH}^\dag(a)\coloneqq\sum_{n\geqslant1}\sum_{[a_1,\ldots,a_n]=a}\frac{1}{n!}a_1\ldots a_n.
\end{equation*}
See Hoffman~\cite{H} for more details. 
Note $\exp_{\rH}^\dag$ is a homomorphism for the concatenation product.
Its inverse 
$\log_{\rH}^\dag\colon\Kb\la\Ab\ra_{\mathrm{conc},\delta}\to\Kb\la\Ab\ra_{\mathrm{conc},\Delta'}$ 
for any letter $a$ from the alphabet as above is 
\begin{equation*}
\log_{\rH}^\dag(a)\coloneqq\sum_{n\geqslant1}\sum_{[a_1,\ldots,a_n]=a}\frac{(-1)^{n-1}}{n}a_1\ldots a_n.
\end{equation*} 
We note $\Kb\la\Ab\ra_{\sh,\Delta}$ and $\Kb\la\Ab\ra_{\mathrm{conc},\delta}$ 
are dual Hopf algebras as well as
$\Kb\la\Ab\ra_{\quas,\Delta}$ and $\Kb\la\Ab\ra_{\mathrm{conc},\Delta'}$. 

\section{Exponential Lie series for continuous semimartingales}\label{sec:Lie}
We show the logarithm of the flowmap for a system of stochastic differential equations,
driven by a set of $d$ continuous semimartingales $X^i_t$ and governed
by an associated set of non-commuting vector fields $V_i$, for $i=1,\ldots,d$, 
can be expressed as a Lie series. 
We begin by emphasizing that for a given set of orthogonal continuous semimartingales,
the only non-zero quadratic variations are those of the form $[X^i,X^i]$ for $i=1,\ldots,d$.
In particular all third order $[X^i,X^i,X^i]$ and thus higher order variations are zero.
Hence the generator $[\,\cdot\,,\,\cdot\,]$, for example in the underlying quasi-shuffle 
algebra, is nilpotent of degree $3$. 
\begin{remark}
Examples of continuous (local) martingales are Brownian motion, time-changed Brownian motion 
and stochastic integrals of Brownian motion. The representation results
of Doob, of Dambis and of Dubins and Schwarz, and of Knight 
(see Theorems 3.4.2, 3.4.6 and 3.4.13 in Karatzas and Shreve~\cite{KS})
show that these examples are fundamental. 
Hence important examples for continuous semimartingales  
are those just mentioned to which a continuous finite variation process, 
i.e.~the difference of two real-valued continuous increasing processes, is added.
\end{remark}
We saw in Section~\ref{sec:flowmap} that the It\^o flowmap has the form $\sum_wI_wD_w$
where the sum is over all words constructed from the alphabet $\Ab$ which contains
the letters $1,\ldots,d$ as well as the letters $[1,1],\ldots,[d,d]$. 
We proposed there, to represent the It\^o flowmap by the abstract expression
$\sum_w w\otimes w$ in $\Kb\la\Ab\ra_{\quas}\otimes\Kb\la\Ab\ra$, the tensor algebra of the
Hopf quasi-shuffle and concatenation algebras. We now make this more precise.
Let $\mathbb I$ denote the algebra generated by multiple It\^o integrals 
$I_w=I_w(t)$ with respect to the semimartingales $\{X^1,\ldots,X^d\}$ or any nested
quadratic variation processes generated from them, and constant random variable $1$.
\begin{definition}[It\^o word-to-integral map]
We denote by $\mu\colon\Kb\la\Ab\ra_{\quas}\to\mathbb I$ the word-to-integral map
$\mu\colon w\mapsto I_w$ assigning each word $w\in\Ab^\ast$ to the corresponding
multiple It\^o integral $I_w$.
\end{definition}
The It\^o word-to-integral map is a quasi-shuffle algebra homomorphism, 
i.e.\/ we have $\mu(u\,\quas\, v)=\mu(u)\mu(v)$ for any $u,v\in\Ab^\ast$; 
see Curry, Ebrahimi--Fard, Malham and Wiese~\cite{CE-FMW}.
Let $\mathbb D$ denote the algebra of scalar linear 
partial differential operators that can be constructed by 
composition from the partial differential operators $D_i$.
\begin{definition}[It\^o word-to-partial differential operator map]
We denote by $\bar{\mu}\colon\Kb\la\Ab\ra\to\mathbb D$ the letter-to-partial 
differential operator map $\bar\mu\colon i\mapsto D_i$ assigning each letter 
$i\in\Ab$ to the corresponding operator $D_i$. Recall the 
operators $D_i$ are given, for $i=1,\ldots,d$, by
$D_i\coloneqq V_i\cdot\pa$ and $D_{[i,i]}\coloneqq \tfrac12 V_{i}\otimes V_{i}\colon\pa^2$.
\end{definition}
The map $\bar{\mu}$ is a concatenation algebra homomorphism, i.e.\/ 
we have $\bar{\mu}(uv)=\bar{\mu}(u)\bar{\mu}(v)$ for any $u,v\in\Ab^\ast$.
Naturally 
$\mu\otimes\bar{\mu}\colon\Kb\la\Ab\ra_\quas\otimes\Kb\la\Ab\ra\to\mathbb I\otimes\mathbb D$
is also an algebra homomorphism. With this homomorphism construction in place, 
we observe
\begin{equation*}
\sum_w I_w\otimes D_w=(\mu\otimes\bar{\mu})\circ\biggl(
\sum_w w\otimes w\biggr).
\end{equation*}
In principle we can compute
the logarithm of $\sum_w w\otimes w$ in $\Kb\la\Ab\ra_\quas\otimes\Kb\la\Ab\ra$
in the form of the quasi-shuffle logarithm as outlined in
Section~\ref{sec:algebra}. 
However the basis in which we expand the flowmap and its logarithm
corresponds, on the right hand side, to the terms $D_w$ in $\mathbb D$ which are
compositions of the vector fields $V_i\cdot\pa$ and 
second order partial differential operators $\tfrac12 V_{i}\otimes V_{i}\colon\pa^2$ 
for $i=1,\ldots,d$. The question is, how can we express the 
logarithm of the flowmap in Lie polynomials or in particular, 
in Lie brackets of vector fields? The natural resolution 
is to use the Fisk--Stratonovich representation of the flowmap.
From the stochastic analysis perspective the procedure is as follows. 
\begin{definition}[Fisk--Stratonovich integral]
For continuous semimartingales $H$ and $Z$, the 
Fisk--Stratonovich integral is defined as 
\begin{equation*}
\int_0^tH_\tau\strat\rd Z_\tau\coloneqq \int_0^tH_\tau\,\rd Z_\tau+\tfrac12[H,Z]_t,
\end{equation*}
where the `$\strat$' indicates Fisk--Stratonovich integration; 
see Protter~\cite[p.~216]{Protter}.
\end{definition}
\begin{lemma}[It\^o to Fisk--Stratonovich conversion]\label{lemma:conv}
For $i=1,\ldots,d$ and any function $f\colon\R^N\to\R$, for 
the integral term in the It\^o chain rule we have
\begin{align*}
\int_0^t\bigl(V_i\cdot\pa\bigr)f(Y_\tau)\,\rd X^i_\tau
&=\int_0^t\bigl(V_i\cdot\pa\bigr)f(Y_\tau)\strat\rd X^i_\tau
-\tfrac12\bigl[\bigl(V_i\cdot\pa\bigr)f(Y),X^i\bigr]_t,
\intertext{and}
\bigl[\bigl(V_i\cdot\pa\bigr)f(Y),X^i\bigr]_t
&=\int_0^t\bigl(V_i\cdot\pa\bigr)\bigl(V_i\cdot\pa\bigr)f(Y_\tau)\,\rd\,[X^i,X^i]_\tau.
\end{align*}
\end{lemma}
\begin{proof}
To establish the first result we set $H=\bigl(V_i\cdot\pa\bigr)f(Y)$
and $Z=X^i$ in the definition for Fisk--Stratonovich integrals above.
For the second result we use the It\^o chain rule in Section~\ref{sec:flowmap} to
substitute for $\bigl(V_i\cdot\pa\bigr)f(Y_t)$ into the quadratic covariation bracket
on the left. Then using that the bracket
is nilpotent of degree $3$ for continuous semimartingales, and zero if any argument 
is constant, establishes the result. 
\qed
\end{proof}
Substituting the product rule
\begin{equation*}
\bigl(V_i\cdot\pa\bigr)\bigl(V_i\cdot\pa\bigr)=
(V_i\otimes V_i)\colon\pa^2+(V_i\cdot\pa V_i)\cdot\pa
\end{equation*}
into the second result and that itself into the first result 
in Lemma~\ref{lemma:conv}, generates the Fisk--Stratonovich chain rule.
\begin{corollary}[Fisk--Stratonovich chain rule]
For any function $f\colon\R^N\to\R$ we have
\begin{equation*}
f(Y)\!=\!f(Y_0)+\!\sum_{i=1}^d\int_0^t\!\!\bigl(V_i\cdot\pa\bigr)f(Y_\tau)\strat\!\rd X^i_\tau
-\tfrac12\!\sum_{i=1}^d\!\int_0^t\!\!\bigl((V_i\cdot\pa V_i)\cdot\pa\bigr)f(Y_\tau)\rd[X^i,X^i]_\tau.
\end{equation*}
\end{corollary}
We emphasize, the differential operator in second term 
on the right is a vector field. In addition, 
the usual rules of calculus apply to Fisk--Stratonovich integrals. 
As in the Section~\ref{sec:flowmap} for the It\^o case, we extend the driving 
continuous semimartingales and governing vector fields as follows.
For $i=1,\ldots,d$ we set $X^{[i,i]}\coloneqq[X^{i},X^{i}]$ as before,
however we now set $V_{[i,i]}\coloneqq-\tfrac12\bigl(V_{i}\cdot\pa V_{i})\cdot\pa$.
Then the chain rule above generates the 
Fisk--Stratonovich representation for the flowmap
\begin{equation*}
\sum_w J_wV_w,
\end{equation*}
where the sum is over all words $w$ constructed from the alphabet given by  
$\Ab\coloneqq\{1,\ldots,d,[1,1],\ldots,[d,d]\}$. 
The multiple stochastic Fisk--Stratonovich integrals 
$J_w=J_w(t)$ are defined over the same simplex as for the multiple It\^o integrals 
but with each nested integration interpreted in the Fisk--Stratonovich sense. 
The basis terms $V_w$ are compositions of vectors fields from the 
alphabet $\Ab$ with the assignment of each letter to each vector
field as outlined above. 
We define the Fisk--Stratonovich word-to-integral map $\nu\colon\Kb\la\Ab\ra_{\sh}\to\mathbb J$ 
by $\nu\colon w\mapsto J_w$, in a similar manner to that for the It\^o word-to-integral map. 
Here $\mathbb J$ denotes the algebra generated by multiple Fisk--Stratonovich integrals 
with respect to the semimartingales $\{X^1,\ldots,X^d\}$ or any 
quadratic variation processes $[X^i,X^i]$ generated from them, and constant random variable $1$.
Note the alphabet $\Ab^\ast$ 
contains all the letters $1,\ldots,d$ and $[1,1],\ldots,[d,d]$.
Let $\mathbb V$ denote the set of partial differential operators 
constructed by composition from the vector fields $V_i$ and $V_{[i,i]}$ 
for $i=1,\ldots,d$.
\begin{definition}[Word-to-vector field map]
We denote by $\bar{\nu}\colon\Kb\la\Ab\ra\to\mathbb V$ the letter-to-vector field map
$\bar\nu\colon i\mapsto V_i$ assigning each letter 
$i\in\Ab$ to the corresponding vector field $V_i$ and each letter $[i,i]\in\Ab$
to the corresponding vector field $V_{[i,i]}$, for $i=1,\ldots,d$.  
\end{definition}
The Fisk--Stratonovich word-to-integral map $\nu$ is a shuffle algebra homomorphism
while the word-to-vector field map $\bar\nu$ is a concatenation algebra homomorphism.
Hence the natural abstract setting for the Fisk--Stratonovich representation
of the flowmap is the complete tensor algebra 
$\Kb\la\Ab\ra_{\sh}\overline{\otimes}\,\Kb\la\Ab\ra$. The map 
$\nu\otimes\bar\nu\colon\Kb\la\Ab\ra_{\sh}\overline{\otimes}\,
\Kb\la\Ab\ra\to\mathbb J\otimes\mathbb V$
is an algebra homomorphism. 

The Fisk--Stratonovich and It\^o representations for the flowmap must
coincide and so must their logarithms. 
The Fisk--Stratonovich integrals $J_w$ satisfy the usual
rules of calculus, see Protter~\cite{Protter}, while the basis terms $V_w$ 
are compositions of vectors fields. Hence the Chen--Strichartz formula applies 
to the Fisk--Stratonovich representation for the flowmap, 
see Strichartz~\cite{S}. Recall from Lemma~\ref{lemma:logpowerseries}
the quasi-shuffle convolution logarithm of the identity $\log^\quas(\id)$
can be expressed as a power series in the augmented ideal projector $J$.
When the quasi-shuffle convolution product reduces to the  
shuffle convolution product, also denoted $\sh$, the 
corresponding shuffle convolution power series 
$\log^{\scriptsize\sh}(\id)$ is given by 
$J-\frac12J^{\scriptsize\sh\,2}+\frac13J^{\scriptsize\sh\,3}+\cdots$. 
\begin{theorem}[Chen--Strichartz Lie series]
The logarithm of the flowmap has the series representation 
\begin{equation*}
\log\biggl(\sum_w J_wV_w\biggr)=\sum_w \frac{1}{|w|}J_{\log^{\scriptsize\sh}(w)}V_{[w]_{\cL}}
\end{equation*}
where $J_{\log^{\scriptsize\sh}(w)}=\nu\circ\log^{\scriptsize\sh}(w)$ and
\begin{equation*}
\log^{\scriptsize\sh}(w)
=\sum_{\sigma\in\Sb_{|w|}}c_\sigma\,\sigma^{-1}(w)
\qquad\text{with}\qquad
c_\sigma\coloneqq\frac{(-1)^{\mathrm{d}(\sigma)}}{|\sigma|}
\begin{pmatrix} |\sigma|-1 \\ \mathrm{d}(\sigma) \end{pmatrix}^{-1}.
\end{equation*}
Here for any word $w=a_1\cdots a_n\in\Ab^\ast$ the basis terms are given by
$V_{[w]_{\cL}}\coloneqq[V_{a_1},[V_{a_2},\ldots,[V_{a_{n-1}},V_{a_n}]_{\cL}\cdots]_{\cL}]_{\cL}$,
where $[\,\cdot\,,\,\cdot\,]_{\cL}$ is the Lie bracket, 
are Lie polynomials. The non-negative integers
$\mathrm{d}(\sigma)$ denote the number of descents in the permutation $\sigma$.
\end{theorem}
\begin{proof} 
First, since the Fisk--Stratonovich multiple integrals satisfy the usual rules of 
calculus, the underlying product is the shuffle product $\sh\,$; this is a special case
of the quasi-shuffle product for which the quadratic variation generator 
$[\,\cdot\,,\,\cdot\,]$ is identically zero. Hence we can emulate the derivation
of the quasi-shuffle convolution given in Section~\ref{sec:algebra} to show that 
\begin{align*}
\log\biggl(\sum J_wV_w\bigr)&=\log\circ(\nu\otimes\bar\nu)\circ\biggl(\sum w\otimes w\biggr)\\
&=(\nu\otimes\bar\nu)\circ\log\circ\biggl(\sum w\otimes w\biggr)\\
&=(\nu\otimes\bar\nu)\circ\biggl(\sum \log^{\scriptsize\sh}(\id)\circ w\otimes w\biggr),
\end{align*}
where the sums are over all words $w\in\Ab^\ast$. 
The expression $\log^\sh(\id)$ is the shuffle convolution power series 
for the logarithm on the identity described above.
That this power series has the equivalent expansion in terms of inverse 
permutations with the form for the coefficients $c_\sigma$ shown, 
is proved in Section~\ref{sec:Chen--Strichartz} for $\log^{\quas}(\id)$. 
See in particular Corollary~\ref{cor:logonwordscomb}.  
Note that the coefficients $J_{\log^{\scriptsize\sh}(w)}$ then refer to the 
linear combination of multiple Fisk--Stratonovich integrals enumerated by 
the words generated by the set of permutations shown.

Second, we express the logarithm of the 
Fisk--Stratonovich flowmap in terms of Lie polynomials as shown.
The crucial observation here is that the adjoint of the shuffle convolution logarithm 
$\log^{\scriptsize\sh}(\id)$ is the concatenation convolution logarithm 
on the identity (computed with the deshuffle coproduct $\delta$) 
which we denote by $\log(\id)$. 
This follows as for any integer $k\geqslant1$, 
the adjoint of $J^{\scriptsize{\sh} k}$ is the $k$th concatenation convolution
power of $J$. The concatenation logarithm of the identity 
is a Lie idempotent known as the Eulerian or Solomon idempotent. 
Another Lie idempotent is the Dynkin idempotent which
has several characterizations, here it suffices to define it as follows. 
Let $[1\cdots p]_{\cL}$ denote left to right Lie bracketing
of the word $1\cdots p$ so that 
$[1\cdots p]_{\cL}\coloneqq[1,[2,[\ldots[p-1,p]\ldots]]]_{\cL}$. 
The element $\frac{1}{p}[1\cdots p]_{\cL}$ of $\Kb[\Sb_p]$ is known as the
Dynkin idempotent, where $\Kb[\Sb_p]$ is the group algebra over $\Kb$
for the symmetric group $\Sb_p$ of order $p$. 
We denote the Dynkin idempotent as
\begin{equation*}
\theta\coloneqq\frac{1}{p}[1\cdots p]_{\cL}.
\end{equation*}
That the Eulerian and Dynkin operators are Lie idempotents is proved
in Reutenauer~\cite[p.~195]{R} for example. 
The Dynkin--Friedrichs--Specht--Wever Theorem is now key. 
It states any polynomial $P$ in the concatenation Hopf algebra $\Kb\la\Ab\ra$
lies in the corresponding free Lie algebra if and only if $\theta P=P$.
The image of the Eulerian idempotent is contained in the free Lie algebra associated 
with $\Kb\la\Ab\ra$; see Reutenauer~\cite[p.~59]{R}.
Thus the element $\log(w)$ for any word $w\in\Kb\la\Ab\ra$ is a Lie element. 
The Dynkin--Friedrichs--Specht--Wever Theorem thus implies
\begin{equation*}
\theta\,\bigl(\log(w)\bigr)\equiv\log(w). 
\end{equation*}
Now by direct calculation,
\begin{align*}
\log\biggl(\sum w\otimes w\biggr)
&=\sum \log^{\scriptsize\sh}(w)\otimes w\\
&=\sum w\otimes\log(w)\\
&=\sum w\otimes\theta\bigl(\log(w)\bigr)\\
&=\sum \log^{\scriptsize\sh}(w)\otimes\theta(w)\\
&=\sum \frac{1}{|w|}\log^{\scriptsize\sh}(w)\otimes[w]_{\cL},
\end{align*}
where the sums are over all words $w\in\Ab^\ast$. Here we used that
if $X$ and $Y$ are adjoint endomorphisms, and $Z$ is an endomorphism
on the concatenation Hopf algebra, then we have 
\begin{equation*}
\sum X(w)\otimes Z(w)=\sum w\otimes Z\bigl(Y(w)\bigr).
\end{equation*}
We applied this identity with $X=\log^{\scriptsize\sh}(\id)$, 
$Y=\log(\id)$ and $Z=\theta$.
\qed
\end{proof}
The Hoffman exponential map $\exp_{\rH}$ naturally relates
It\^o and Fisk--Stratonovich multiple integrals as follows; see Kloeden and 
Platen~\cite[Remark 5.2.8]{KP} for the Wiener process case.
\begin{proposition}[It\^o to Fisk--Stratonovich: Hoffman exponential]\label{prop:Hoffexp}
For continuous semimartingales, for any word $w$ in $\Ab^\ast$, we have 
$J_w=I_{\exp_{\rH}(w)}$ where explicitly we have 
\begin{equation*}
I_{\exp_{\rH}(w)}
=\sum_{\lambda\in\CC(|w|)}\frac{1}{2^{\Sigma(\lambda)-|\lambda|}}I_{\lambda\circ w}
=I_w+\sum_{u\in[[w]]}\frac{1}{2^{|w|-|u|}}I_{u} .
\end{equation*}
The nilpotency of the generator $[\,\cdot\,,\,\cdot\,]$ implies 
the compositions $\lambda\in\CC(|w|)$ with a nonzero contribution 
only contain letters $1$ and $2$
and so $\Gamma(\lambda)=2^{\Sigma(\lambda)-|\lambda|}$. Thus 
the set $[[w]]$ consists of the words we can construct from $w$ by successively
replacing any neighbouring pairs $ii$ in $w$ by $[i,i]$. 
\end{proposition}
\begin{proof}
Using the definition of the Fisk--Stratonovich integral,
for any word $w=a_1\cdots a_n$ we have
\begin{equation*}
J_{a_1\cdots a_n}=\int J_{a_1\cdots a_{n-1}}\,\rd I_{a_n}+\frac12\int J_{a_1\cdots a_{n-2}}\,\rd[I_{a_{n-1}},I_{a_n}].
\end{equation*}
Recursively applying this formula for $J_{a_1\cdots a_{n-1}}$ 
and so forth, generates the result.
\qed
\end{proof}
With this in hand, we deduce the following main result of this section.
\begin{corollary}[It\^o Lie series]\label{cor:ItoLieseries}
Let $c_\sigma$ denote the coefficient of $\sigma^{-1}(w)$ in the 
expression for $\log^{\scriptsize\sh}(w)$ above. We can express
the Chen--Strichartz Lie series in terms of multiple It\^o
integrals as follows 
\begin{equation*}
\log\biggl(\sum_w J_wV_w\biggr)
=\sum_w\sum_{\sigma\in\Sb_{|w|}}\frac{c_\sigma}{|\sigma|}
I_{\exp_{\rH}(\sigma^{-1}(w))}V_{[w]_{\cL}},
\end{equation*}
or equivalently, by resummation of the series,
\begin{equation*}
\log\biggl(\sum_w J_wV_w\biggr)
=\sum_wI_{w}\Biggl(\sum_{\sigma\in\Sb_{|w|}}
\frac{c_\sigma}{|\sigma|}V_{[\sigma(w)]_{\cL}}
+\sum_{u\in]]w[[}\frac{1}{2^{|u|-|w|}}\sum_{\sigma\in\Sb_{|u|}}\frac{c_\sigma}{|\sigma|}
V_{[\sigma(u)]_{\cL}}\Biggr).
\end{equation*}
Here for any word $w\in\Ab^\ast$ the set $]]w[[$ consists of $w$
and all words we construct from $w$ by successively replacing
any letter $[i,i]$ with $ii$, for $i=1,\ldots,d$.
\end{corollary}
Since the Fisk--Stratonovich and It\^o representations for the flowmap must
coincide, the expression above must coincide with the logarithm of the 
It\^o representation for the flowmap. 
The key fact distinguishing the It\^o from the 
Fisk--Stratonovich representation for the flowmap is that
the word-to-vector field map 
$\bar\nu\colon\Kb\la\Ab\ra_{\mathrm{conc},\delta}\to\mathbb V$ 
and word-to-partial differential operator map 
$\bar\mu\colon\Kb\la\Ab\ra_{\mathrm{conc},\Delta'}\to\mathbb D$, 
which are both concatenation homomorphisms, assign
\begin{equation*}
\bar\mu\colon
\begin{cases}
~~i&\!\!\mapsto V_i\cdot\pa,\\
[i,i]&\!\!\mapsto\tfrac12V_i\otimes V_i\colon\pa^2,
\end{cases}
\qquad\text{and}\qquad
\bar\nu\colon
\begin{cases}
~~i&\!\!\mapsto V_i\cdot\pa,\\
[i,i]&\!\!\mapsto-\tfrac12(V_i\cdot\pa V_i)\cdot\pa.
\end{cases}
\end{equation*}
The It\^o representation is constructed by composing 
the operators on the left shown above, while that for the 
Fisk--Stratonovich representation is constructed by composing
the operators on the right. The operators 
$\tfrac12V_i\otimes V_i\colon\pa^2$ and $\tfrac12(V_i\cdot\pa V_i)\cdot\pa$ 
are both associated with the quadratic variation process $[X^i,X^i]$. 
The It\^o word-to-integral map 
$\mu\colon\Kb\la\Ab\ra_{\quas,\Delta}\to\mathbb I$
and Fisk--Stratonovich word-to-integral map 
$\nu\colon\Kb\la\Ab\ra_{\sh,\Delta}\to\mathbb J$
both assign $i\mapsto X^i$ and $[i,i]\mapsto[X^i,X^i]$.
The former is a quasi-shuffle homomorphism and
the latter a shuffle homomorphism and consequently
$J_w=I_{\exp_{\rH}(w)}$. This relation, together with
the calculus product rule given by
$\bigl(V_i\cdot\pa\bigr)\bigl(V_i\cdot\pa\bigr)=
(V_i\otimes V_i)\colon\pa^2+(V_i\cdot\pa V_i)\cdot\pa$
underlie the following result. Recall the definition 
of the Hoffman exponential and logarithm map adjoints 
in Section~\ref{sec:algebra}.
\begin{theorem}[It\^o and Fisk--Stratonovich map relations]
The two word-to-integral maps $\mu$ and $\nu$ and the 
word-to-vector field and word-to-partial differential operator maps
$\bar\mu$ and $\bar\nu$ are related as follows
\begin{equation*}
\nu\equiv\mu\circ\exp_{\rH}
\qquad\text{and}\qquad
\bar\nu\equiv\bar\mu\circ\log_{\rH}^\dag.
\end{equation*}
\end{theorem}
\begin{proof}
The first relation follows directly from 
$J_w=I_{\exp_{\rH}(w)}~\Leftrightarrow~\nu\circ w=\mu\circ\exp_{\rH}\circ w$. 
The second relation follows from 
the calculus product rule above which can be expressed
in the form
\begin{equation*}
\tfrac12\bar\mu\circ ii=\bar\mu\circ[i,i]-\bar\nu\circ[i,i]
\Leftrightarrow
\bar\nu\circ[i,i]=\bar\mu\circ\bigl([i,i]-\tfrac12ii\bigr)
\Leftrightarrow
\bar\nu\circ[i,i]=\bar\mu\circ\log_{\rH}^\dag\circ[i,i].
\end{equation*}
Note $ii$ denotes the concatenation of $i$ with $i$ and $\bar\mu(ii)=\bar\nu(ii)$. 
The final expression in this sequence follows using the nilpotency of the bracket
$[\,\cdot\,,\,\cdot\,]$. 
Using that $\log_{\rH}^\dag\circ i=i$ for the letters $i=1,\ldots,d$ and
$\log_{\rH}^\dag$ is a concatenation homomorphism, establishes the second result.
\qed
\end{proof}
Some immediate consequences of this result are as follows. First, algebraically, we observe
\begin{align*}
\sum I_wD_w
&=\sum \mu\circ w\otimes\bar\mu\circ w\\
&=\sum \mu\circ w\otimes\bar\nu\circ\exp_{\rH}^\dag\circ w\\
&=(\mu\otimes\bar\nu)\circ\biggl(\sum w\otimes\exp_{\rH}^\dag\circ w\biggr)\\
&=(\mu\otimes\bar\nu)\circ\biggl(\sum \exp_{\rH}\circ w\otimes w\biggr)\\
&=\sum \mu\circ\exp_{\rH}\circ w\otimes\bar\nu\circ w\\
&=\sum \nu\circ w\otimes\bar\nu\circ w\\
&=\sum J_wV_w.
\end{align*}
In other words the It\^o and Fisk--Stratonovich flowmaps coincide. Note
the transfer of $\exp_{\rH}^\dag$ from the right of the tensor product to 
$\exp_{\rH}$ on the left, relies solely on the vector space properties
of $\Kb\la\Ab\ra$. Second, using this result we observe 
\begin{align*}
\log\biggl(\sum I_wD_w\biggr)
&=\log\circ(\mu\otimes\bar\mu)\circ\biggl(\sum w\otimes w\biggr)\\
&=\log\circ(\nu\otimes\bar\nu)\circ\biggl(\sum w\otimes w\biggr)\\
&=(\nu\otimes\bar\nu)\circ\log\circ\biggl(\sum w\otimes w\biggr)\\
&=(\nu\otimes\bar\nu)\circ\biggl(\sum\log^{\scriptsize\sh}\circ w\otimes w\biggr)\\
&=(\nu\otimes\bar\nu)\circ\biggl(\sum\frac{1}{|w|}\log^{\scriptsize\sh}\circ w\otimes [w]_{\cL}\biggr)\\
&=(\mu\otimes\bar\nu)\circ\biggl(
 \sum\frac{1}{|w|}\exp_{\rH}\circ\log^{\scriptsize\sh}\circ w\otimes [w]_{\cL}\biggr)\\
&=\sum \frac{1}{|w|}I_{\exp_{\rH}(\log^{\scriptsize\sh}(w))}V_{[w]_{\cL}},
\end{align*}
where the sums are over all words $w\in\Ab^\ast$. We can also perform a 
resummation of the series as indicated in Corollary~\ref{cor:ItoLieseries}
as follows,
\begin{align*}
\log\biggl(\sum J_wV_w\biggr)
&=(\nu\otimes\bar\nu)\circ\biggl(\sum\log^{\scriptsize\sh}\circ w\otimes \theta\circ w\biggr)\\
&=(\mu\otimes\bar\nu)\circ\biggl(
 \sum\exp_{\rH}\circ\log^{\scriptsize\sh}\circ w\otimes\theta\circ w\biggr)\\
&=(\mu\otimes\bar\nu)\circ\biggl(
 \sum w\otimes\theta\circ\log\circ\exp_{\rH}^\dag\circ w\biggr),
\end{align*}
where 
\begin{align*}
\log\circ\exp_{\rH}^\dag\circ w&=\log\circ\Biggl(w+\sum_{u\in]]w[[}\frac{1}{2^{|u|-|w|}}u\Biggr)\\
&=\sum_{\sigma\in\Sb_{|w|}}c_\sigma\,\sigma(w)
+\sum_{u\in]]w[[}\frac{1}{2^{|u|-|w|}}\sum_{\sigma\in\Sb_{|u|}}c_\sigma\,\sigma(u).
\end{align*}
These last two results are
thus a restatement of the It\^o Lie series results in Corollary~\ref{cor:ItoLieseries}. %therein.
\begin{remark} 
From the algebraic combinatorial computations above,
we observe: (1) In the first computation above the 
transformation from the It\^o to Fisk--Stratonovich
flowmaps was instigated by the transformation of coordinates $\mu=\nu\circ\exp_{\rH}^\dag$.
In other words this transformation, which is a direct result of the product rule,
encodes all the information required for It\^o to Fisk--Stratonovich conversion;
(2) Quadratic variations are a natural component in the Fisk--Stratonovich
formulation; (3) The encoding which retains the letters $[1,1],\ldots,[d,d]$
as well as $1,\ldots,d$ in the alphabet appears to be natural, 
especially in the context of using the quasi-shuffle machinery 
provided by Hoffman~\cite{H}. Indeed this is also the case for
stochastic differential equations driven by Wiener processes for
which it is usual to replace the quadratic variation terms by the 
corresponding drift term; 
(4) The flowmap satisfies the linear equation 
$\varphi=\id+\int\varphi\,\mathrm{d}S$ with $S\coloneqq\sum_i D_iX^i$, 
see Ebrahimi-Fard, Malham, Patras and Wiese~\cite{E-FMPW}. 
When the coefficients $D_i$ are constant, the solution is the 
well-known Dol{\'e}ans-Dade exponential; a representation of it in terms 
of iterated integrals was derived in Jamishidian~\cite{J2011}.
\end{remark}

\section{Quasi-Shuffle Chen--Strichartz formula}\label{sec:Chen--Strichartz}
In this section we do not make any nilpotency assumptions 
on $n$-fold nested brackets as in Section~\ref{sec:Lie}. 
We derive an explicit formula for the coefficients of the
quasi-shuffle convolution logarithm of the identity endomophism on $\Kb\la\Ab\ra_\quas$,
i.e.\/ we explicitly enumerate
\begin{equation*}
\log^\quas(\id).
\end{equation*}
This represents the quasi-shuffle logarithm equivalent of the 
Chen--Strichartz shuffle logarithm formula and was derived
in Novelli, Patras and Thibon~\cite{NPT} and generalized to 
linear matrix valued systems in 
Ebrahimi--Fard, Malham, Patras and Wiese~\cite{E-FMPW}.
Using the notion of surjections instead of permutations and
quasi-descents, we can closely follow the 
development given in Reutenauer~\cite{R}.
We begin by outlining the theory of surjections and 
quadratic covariation permutations, which we hereafter
call ``quasi-permutations'', as well as their action on words. 
We denote the symmetric group of order $p$ by
$\mathbb S_p$ and corresponding group algebra over the field $\Kb$ by $\Kb[\mathbb S_p]$.
The crucial fact about any permutation $\sigma\in\Kb[\mathbb S_p]$, which underlies
the classical shuffle Chen--Strichartz formula, is that 
the inverse $\sigma^{-1}$ records the following information:
``The letter $i$ is at position $\sigma^{-1}(i)$ in $\sigma$''. 
We exploit the corresponding result for surjections herein. 

We shall denote the set of surjective maps from the set of natural numbers 
$\{1,\ldots,p\}$ to the set of natural numbers $\{1,\ldots,q\}$ with $q\leqslant p$
by $\mathbb S_{p,q}'$. We set
\begin{equation*}
\mathbb S_p'\coloneqq\bigcup_{q\leqslant p}\mathbb S_{p,q}'.
\end{equation*}
Naturally we have $\mathbb S_p\subseteq\mathbb S_p'$. 
Associated with each surjection in $\mathbb S_p'$ 
is a quasi-permutation.
\begin{definition}[Quasi-permutations]
We denote by $\CC(\mathbb S_p)$ the set of all quasi-permutations, 
these are all the permutations in $\mathbb S_p$                               
together with all unique words formed by applying all possible composition                  
actions to these permutations, taking care to unify equivalent terms. In other words,       
\begin{equation*}                                                                           
\CC(\mathbb S_p)\coloneqq                                                                   
\bigl\{\lambda\circ\sigma\colon\sigma\in\Sb_p,\lambda\in\CC(p)\bigr\},
\end{equation*}                                                                                  
where we identify all terms that are equal due to the symmetry and associative                    
properties of the nested bracket operation.                                                      
\end{definition}                                                                                 
Henceforth we record quasi-permutations simply as $\sigma\in\CC(\mathbb S_p)$.
However, it is always possible to decompose (non-uniquely) any given                                            
quadratic covariation permutation into its composition and permutation components,               
say as $\lambda\circ\rho$ or as a pair $(\lambda,\rho)$ where $\lambda$                          
is a composition and $\rho$ a permutation.                                                       
Given any quasi-permutation in $\CC(\mathbb S_p)$,
there is a unique surjection in $\mathbb S_p'$ that
records the position of letters and nested brackets of letters.
\begin{example}
Consider the set of all quasi-permutations $\CC(\Sb_3)$, these are given by
the set of $\Sb_3$ permutations $123$, $132$, $213$, $231$, $312$, $321$ and 
$[1,2]3$, $1[2,3]$, $[1,3]2$, $2[1,3]$, $[2,3]1$, $3[1,2]$, $[1,2,3]$. 
The set of all surjections in $\mathbb S_3'$  
consists of $123$, $132$, $213$, $312$, $231$, $321$ and 
$112$, $122$, $121$, $212$, $211$, $221$, $111$. 
Term by term in the order given, we see that the surjections record the corresponding
positions of the letters in the quasi-permutations. 
\end{example}
Hence, by analogy with permutations, quasi-permutations play the role
of generalized permutations, while the corresponding 
surjections play the role of the inverse permutations by recording the positions of
the letters in the corresponding quasi-permutations. 
Hence we have the corresponding statement to that above and crucial 
fact about surjections: each surjection $\zeta$ corresponding to a 
given quasi-permutation $\sigma$ records the information:
\begin{equation*}
\text{The letter $i$ is at position $\zeta(i)$ in $\sigma$.}
\end{equation*}
\begin{example}
The surjection $3221$ from $\Sb_4'$, tells us that the letter that was in position 
$1$ in the quasi-permutation was sent to position position $3$, the letters $2$ and $3$ 
were sent to position $2$, and the letter $4$ was sent to position~$1$. Hence the 
corresponding quasi-permutation is $4[2,3]1$. For another example, 
if $2312$ is a given surjection in $\Sb_4'$ mapping 
$1\mapsto 2$, $2\mapsto3$, $3\mapsto1$ and $4\mapsto2$,
then the corresponding quasi-permutation is $3[1,4]2$ which is equal to $3[4,1]2$.
\end{example}
With each surjection we can associate a quasi-descent set.
\begin{definition}[Quasi-descent sets]
Given any surjection $\zeta\in\mathbb S_p'$
we define its quasi-descent set $\mathrm{Des}(\zeta)$ 
to be the list of the indices $k\in\{1,\ldots,p-1\}$ 
for which $\zeta(k+1)\leqslant\zeta(k)$.
\end{definition}
For the particular subset $\mathbb S_p\subset\mathbb S_p'$ these inequalities
would be strict and the indices $k$ would correspond to the classical descent indices.
Just as there is an intimate relation between shuffles and descents, there is 
also one between quasi-shuffles and quasi-descents. The following first key result underlies
the whole of this section.
\begin{lemma}[Quasi-descents and quasi-shuffles]\label{lemma:qdqsc} 
The set of surjections $\zeta\in\mathbb S_p'$ satisfying
$\mathrm{Des}(\zeta)\subseteq\{q\}$ for 
$q<p$, is identical to the set of surjections satisfying 
$\zeta(1)<\cdots<\zeta(q)$
and $\zeta(q+1)<\cdots<\zeta(p)$.
\end{lemma}
\begin{proof}
We observe that for any surjection $\zeta\in\mathbb S_p'$
for which $\mathrm{Des}(\zeta)\subseteq\{q\}$ for some natural
number $q<p$, then discounting the case when the quasi-descent set is
empty, by definition we must have 
$\zeta(k)\geqslant\zeta(k+1)\implies k=q~~\Leftrightarrow~~k\neq q\implies\zeta(k)<\zeta(k+1)$.
The latter condition is equivalent to that in the statement 
of the lemma. 
\qed
\end{proof}
The second key result we establish in this section is a natural consequence.
\begin{corollary}[Quasi-descents and quasi-shuffles] \label{cor:qshufdesc}
Let $q<p$ be natural numbers. If we factorize the word $1\cdots p=u_1u_2$
with $|u_1|=q$ and $|u_2|=p-q$, then the quasi-shuffle product of $u_1\quas u_2$ is given by
\begin{equation*}
u_1\,\quas\, u_2
\equiv\sum_{\substack{\zeta(1)<\cdots<\zeta(q)\\ \zeta(q+1)<\cdots<\zeta(p)}}
\sigma(\zeta)
\equiv\sum_{\mathrm{Des}(\zeta)\subseteq\{q\}}\sigma(\zeta),
\end{equation*}
where $\sigma(\zeta)$ denotes the unique quasi-permutation associated with
a given surjection $\zeta$.
The first sum is over all $\zeta\in\mathbb S_p'$ satisfying the inequalities
shown, and the second sum is over all $\zeta\in\mathbb S_p'$
such that $\mathrm{Des}(\zeta)\subseteq\{q\}$.
\end{corollary}
\begin{proof}
Recall the definition of the quasi-shuffle product and its 
generation through the formula 
\begin{align*}
(1\cdots q)\,\quas\,(q+1\cdots p)=&\;\bigl((1\cdots q-1)\,\quas\,(q+1\cdots p)\bigr)\,q\\
&\;+\bigl((1\cdots q)\,\quas\,(q+1\cdots p-1)\bigr)\,p\\
&\;+\bigl((1\cdots q-1)\,\quas\,(q+1\cdots p-1)\bigr)\,[q,p].
\end{align*}
We observe that if we recursively apply this formula to obtain on the righthand side
the complete sum over all quasi-permutations, then the quasi-shuffle product
of $u_1$ and $u_2$ is equivalent to the prescription that it is the sum over all
quasi-permutations whose corresponding surjections satisfy the set of
inequalities $\zeta(1)<\cdots<\zeta(q)$ 
and $\zeta(q+1)<\cdots<\zeta(p)$. This establishes the first result.
With this in hand, the result of the quasi-descent and quasi-shuffle conditions 
Lemma~\ref{lemma:qdqsc} above, implies the equivalence to the second result.
\qed
\end{proof}
The following generalization is then immediate and represents
the quasi-shuffle analog of Lemma~3.13 in Reutenauer~\cite[p.~65]{R}.
\begin{corollary}[Multiple quasi-shuffles and quasi-descents] \label{cor:qshufdesc2}
Let $p_1$, \ldots, $p_k$ be positive integers of sum $p$ 
and $S=\{p_1, p_1+p_2,\ldots,p_1+\cdots+p_{k-1}\}$
be a subset of $\{1,\ldots,p-1\}$. If we factorize the word
$1\cdots p=u_1\cdots u_k$ with $|u_i|=p_i$ for $i=1,\ldots,k$, then 
we have
\begin{equation*}
u_1\,\quas\cdots\quas\,u_k=\sum_{\mathrm{Des}(\zeta)\subseteq S}\sigma(\zeta).
\end{equation*}
\end{corollary}
Much like the symmetric group action on words there is an analogous
quasi-permutation action on words.
Recall we can decompose any quasi-permutation in $\Kb[\CC(\mathbb S_p)]$
as $\sigma=\lambda\circ\rho$, into its composition $\lambda\in\CC$ 
and permutation $\rho\in\mathbb S_p$ components.
\begin{definition}[Quasi-permutation action]
The action of $\Kb[\CC(\mathbb S_p)]$ on 
$\Kb\la\Ab\ra$ for any $\sigma\in\Kb[\CC(\mathbb S_p)]$
decomposed as $\sigma=\lambda\circ\rho$ and word $w=a_1\ldots a_p$, is defined by 
$\sigma w\coloneqq\lambda\circ(a_{\rho(1)}\cdots a_{\rho(p)})$.
\end{definition}

We can now construct the quasi-shuffle logarithm of the identity.
We start with the quasi-shuffle convolution powers of the augmented ideal projector $J$.
\begin{corollary}[Convolution powers and descents]\label{cor:convpowers}
For any $k\geqslant1$ and word $w$ we have
\begin{equation*}
J^{\quas k}(w)=\Biggl(\sum_{|S|=k-1}\sum_{\mathrm{Des}(\zeta)\subseteq S}\sigma(\zeta)\Biggr)\circ w.
\end{equation*}
\end{corollary}
\begin{proof}
For any word $w$ the quantity $J^{\quas k}(w)$ is the sum 
over all possible $k$-partitions of $w$, say $v_1\cdots v_k$, 
quasi-shuffled together. Hence we have
\begin{align*}
J^{\quas k}(w)&=\sum_{v_1\cdots v_k=w}v_1\,\quas\ldots\quas\,v_k\\
&=\sum_{u_1\cdots u_k=1\cdots p}(u_1\,\quas\ldots\quas\,u_k)\circ w\\
&=\sum_{|S|=k-1}\Biggl(\sum_{\mathrm{Des}(\zeta)\subseteq S}\sigma(\zeta)\Biggr)\circ w\\
&=\Biggl(\sum_{|S|=k-1}\sum_{\mathrm{Des}(\zeta)\subseteq S}\sigma(\zeta)\Biggr)\circ w,
\end{align*}
where we used Corollary~\ref{cor:qshufdesc2} in the third step. 
\qed
\end{proof}
\begin{corollary}[Quasi-Shuffle convolution logarithm on words]\label{cor:logonwords}
The action of the quasi-shuffle convolution logarithm 
on any word $w$ is as follows
\begin{equation*}
\log^\quas(w)=\Biggl(\sum_{S\subseteq\{1,\ldots,|w|-1\}}\frac{(-1)^{|S|}}{|S|+1}
\sum_{\mathrm{Des}(\zeta)\subseteq S}\sigma(\zeta)\Biggr)\circ w.
\end{equation*}
\end{corollary}
\begin{proof}
By direct computation using Corollary~\ref{cor:convpowers} we find
\begin{align*}
\log^\quas(w)&=\sum_{k\geqslant 1}\frac{(-1)^{k-1}}{k}J^{\quas k}(w)\\
&=\sum_{k\geqslant 1}\frac{(-1)^{k-1}}{k}
\Biggl(\sum_{|S|=k-1}\sum_{\mathrm{Des}(\zeta)\subseteq S}\sigma(\zeta)\Biggr)\circ w\\
&=\Biggl(\sum_{|S|\geqslant 0}\frac{(-1)^{|S|}}{|S|+1}
\sum_{\mathrm{Des}(\zeta)\subseteq S}\sigma(\zeta)\Biggr)\circ w\\
&=\Biggl(\sum_{S\subseteq\{1,\ldots,|w|-1\}}\frac{(-1)^{|S|}}{|S|+1}
\sum_{\mathrm{Des}(\zeta)\subseteq S}\sigma(\zeta)\Biggr)\circ w.
\end{align*}
\qed
\end{proof}
The following characterization of the quasi-shuffle convolution logarithm is
the generalization of the standard shuffle convolution logarithm.
We shall need the following integral identity
for non-negative integers $d$ and $r$ which is 
proved for example in Reutenauer~\cite[p.~69]{R}:
\begin{equation*}
\int_{-1}^0x^d(1+x)^r\,\mathrm{d}x=\frac{(-1)^dd!r!}{(d+r+1)!}.
\end{equation*}
\begin{corollary}[Quasi-Shuffle convolution logarithm endomorphism]\\\label{cor:logonwordscomb}
The quasi-shuffle convolution logarithm~$\log^\quas(\id)$ acts on 
$1\cdots p$ as follows
\begin{equation*}
\log^\quas(1\cdots p)=\sum_{\zeta\in\Sb_p'}\frac{(-1)^{\mathrm{d}(\zeta)}}{p}
\begin{pmatrix} p-1 \\ \mathrm{d}(\zeta) \end{pmatrix}^{-1}\sigma(\zeta),
\end{equation*}
where $\mathrm{d}(\zeta)$ denotes the number of quasi-descents in $\zeta$.
\end{corollary}
\begin{proof}
We observe from Corollary~\ref{cor:logonwords} that $\log^\quas(1\cdots p)$
consists of a linear combination of quasi-permutations $\sigma(\zeta)$. 
Hence we directly compute the coefficient of an arbitrary quasi-permutation
$\sigma(\zeta)$ in $\log^\quas(1\cdots p)$ which, using the result of
Corollary~\ref{cor:logonwords}, is given by
\begin{equation*}
\sum_{S\subseteq\{1,\ldots,p-1\}\colon\mathrm{Des}(\zeta)\subseteq S}\frac{(-1)^{|S|}}{|S|+1}.
\end{equation*}
Suppose that $\zeta$ has quasi-descent indicies $p_1,\ldots,p_k$ so that
$\mathrm{d}(\zeta)=k$. To compute this coefficient we therefore have to
determine the number of subsets $S\subseteq\{1,\ldots,p-1\}$ which 
contain $p_1,\ldots,p_k$. Note the coefficient itself only depends on the
size of such sets. These subsets have possible size $|S|=k$
through to $|S|=p-1$. Starting with the case $|S|=k$ there is of course
only one set of this size containing $p_1,\ldots,p_k$, the set of
these integers themselves. Now consider the case $|S|=k+1$. Then an
extra ``quasi-descent'' can be placed in total of $p-1-k$ possible
positions, or equivalently in $p-1-k$ choose $1$ ways. 
When $|S|=k+2$, there are $p-1-k$ choose $2$ ways, and so forth so that
in general, when $|S|=k+i$, there are $p-1-i$ choose $i$ possible ways. 
Hence the coefficient above equals 
\begin{equation*}
\sum_{i=0}^{p-1-\mathrm{d}(\zeta)}\begin{pmatrix}p-1-\mathrm{d}(\zeta)\\i \end{pmatrix}
\frac{(-1)^{\mathrm{d}(\zeta)+i}}{\mathrm{d}(\zeta)+i+1}.
\end{equation*}
This form of the coefficient of is equal to 
\begin{align*}
\int_{-1}^0\sum_{i=0}^{p-1-\mathrm{d}(\zeta)}
\begin{pmatrix}p-1-\mathrm{d}(\zeta)\\i \end{pmatrix}x^{\mathrm{d}(\zeta)+i}\,\mathrm{d}x
&=\int_{-1}^0x^{\mathrm{d}(\zeta)}(1+x)^{p-1-\mathrm{d}(\zeta)}\,\mathrm{d}x\\
&=(-1)^{\mathrm{d}(\zeta)}
\frac{\mathrm{d}(\zeta)!(p-1-\mathrm{d}(\zeta))!}{p!}\\
&=\frac{(-1)^{\mathrm{d}(\zeta)}}{p}
\begin{pmatrix}p-1\\\mathrm{d}(\zeta)\end{pmatrix}^{-1},
\end{align*}
using the integral identity preceding the corollary. 
\qed
\end{proof}
\begin{remark}
This is equivalent to the quasi-shuffle logarithm given in 
Ebrahimi--Fard \textit{et al.}\/ \cite[Theorem~6.2]{E-FMPW}.
We included a self-contained derivation here for completeness.
\end{remark}

\section{Concluding remarks}\label{sec:conclu}
There has been a recent surge in the development of 
quasi-shuffle algebras and stochastic Taylor solution formulae 
in the context of semimartingales, on the 
theoretical and practical level. See Platen and Bruti--Liberati~\cite{B-LP},
and Marcus~\cite{Marcus}, Friz and Shekhar~\cite{FS} 
and Hairer and Kelly~\cite{HK} for contemporary references. 
For example Li and Liu~\cite{LL} considered systems driven 
by both Wiener and Poisson processes. 
We have shown that the Chen--Strichartz flowmap solution formula which
is well-known for Stratonovich stochastic differential systems 
driven by Wiener processes extends to systems driven by 
general continuous semimartigales. 
We demonstrated it is in fact a Lie series, and this property holds
irrespective of whether we consider the system in the It\^o
or Stratonovich sense. We also give and prove an explicit
formula for the Lie series coefficients.
Curry, Ebrahimi--Fard, Malham and Wiese~\cite{CE-FMW} have developed so-called efficient
simulation schemes for such systems driven by L\'evy processes. This 
involves the antisymmetric sign reverse endomorphism rather than
the quasi-shuffle logarithm endomorphism.

\begin{acknowledgements}
KEF is supported by Ram\'on y Cajal research grant RYC-2010-06995 from the Spanish government
and acknowledges support from the Spanish government under project MTM2013-46553-C3-2-P.
This research also received support from a grant by the BBVA Foundation.
FP acknowledges support from the grant ANR-12-BS01-0017, Combinatoire Alg\'ebrique,
R\'esurgence, Moules at Applications.
\end{acknowledgements}


\begin{thebibliography}{}

\bibitem{Azencott} Azencott R. 1982.
Formule de Taylor stochastique et d\'eveloppement
asymptotique d'int\'egrales de Feynman.               
\textit{Seminar on Probability XVI, Lecture Notes in Math.}    
\textbf{921}, pp.~237--285.                    

\bibitem{BA} Ben Arous G. 1989.  
Flots et series de Taylor stochastiques.
\textit{Probab. Theory Related Fields}                
\textbf{81}, pp.~29--77.         

\bibitem{Baudoin} Baudoin F. 2004.
\textit{An introduction to the geometry of
stochastic flows}. Imperial College Press.

\bibitem{CG} Castell F, Gaines J. 1995.  
An efficient approximation method for stochastic
differential equations by means of the exponential    
Lie series.                                          
\textit{Math. Comput. Simulation}
\textbf{38}, pp.~13--19.

\bibitem{CG2} Castell F, Gaines J. 1996. The ordinary differential equation
approach to asymptotically efficient schemes for solution
of stochastic differential equations.
\textit{Ann. Inst. H. Poincar\'e Probab. Statist.}
\textbf{32}(2), pp.~231--250.

\bibitem{Chen} Chen KT. 1957.
Integration of paths, geometric invariants and
a generalized Baker--Hausdorff formula.
\textit{Annals of Mathematics} \textbf{65}(1), pp.~163--178.

\bibitem{CE-FMW} Curry C, Ebrahimi--Fard K, Malham SJA,
Wiese A. 2014, L\'evy processes and quasi-shuffle algebras. 
\textit{Stochastics} \textbf{86}(4), pp.~632--642.

\bibitem{CE-FMW:efficient} Curry C, Ebrahimi--Fard K, Malham SJA, Wiese A. 2015.
Algebraic structures and stochastic differential equations driven by L\'evy processes.
In preparation.

\bibitem{E-FG} Ebrahimi--Fard K, Guo L. 2006. Mixable shuffles,
quasi-shuffles and Hopf algebras. \textit{Journal of algebraic combinatorics}
\textbf{24}(1), pp.~83--101   

\bibitem{E-FLMMW} Ebrahimi--Fard K, Lundervold A,  Malham SJA, Munthe--Kaas H,
Wiese A. 2012.
Algebraic structure of stochastic expansions and efficient simulation.
\textit{Proc. R. Soc. A} doi:10.1098/rspa.2012.0024.

\bibitem{E-FMPW} Ebrahimi--Fard K, Malham SJA, Patras F,
Wiese A. 2015. Flows and stochastic Taylor series in It\^o calculus,
submitted.

\bibitem{EM} Eilenberg, S, Mac Lane, S. 1953.
On the groups $H(\pi,n)$. 
\textit{Annals of Mathematics} \textbf{58}(1), pp.~55--106.

\bibitem{Fliess} Fliess M. 1981. 
Functionelles causales non lin\'eaires et ind\'etermin\'ees non-commutatives.                      
\textit{Bulletin de la Soci\'et\'e Math\'ematique de France} 
\textbf{109}, pp.~3--40.

\bibitem{FPT} Foissy L, Patras F, Thibon J-Y. 2013. 
Deformations of shuffles and quasi-shuffles. \textit{Ann. Inst. Fourier}, 
to appear. arXiv:1311.1464v1.

\bibitem{FS} Friz PK, Shekhar A. 2014.
General Rough integration, L\'evy Rough paths and a L\'evy--Kintchine type formula.
arXiv:1212.5888v2.

\bibitem{G} Gaines JG. 1994.
The algebra of iterated stochastic integrals.
\textit{Stochastics and stochastics reports}
\textbf{49}(3--4), pp.~169--179.

\bibitem{G2} Gaines JG. 1995.
A basis for iterated stochastic integrals
\textit{Math. Comput. Simulation}
\textbf{38}, pp.~7--11.

\bibitem{HK} Hairer M, Kelly D. 2015. 
Geometric versus non-geometric rough paths.
\textit{Ann. Inst. H. Poincar\'e Probab. Statist.}
\textbf{51}(1), pp.~207--251.

\bibitem{H} Hoffman ME. 2000. Quasi-shuffle products. 
\textit{Journal of Algebraic Combinatorics} \textbf{11}, pp.~49--68.

\bibitem{HI} Hoffman ME, Ihara K. 2012. Quasi-shuffle products revisted.
Max-Planck-Institut f\"ur Mathematik Preprint Series 2012 (16).

\bibitem{J2011} Jamshidian F. 2011. On the combinatorics of iterated stochastic
integrals. \textit{Stochastics} \textbf{83}, pp.~1--15.

\bibitem{KS} Karatzas I, Shreve S. 2000. 
\textit{Brownian Motion and Stochastic Calculus}, 
Graduate Texts in Mathematics Vol. 113, 2nd Edition, Springer.

\bibitem{KP} Kloeden PE, Platen E. 1999. 
\textit{Numerical solution of stochastic
differential equations}, Springer.
 
\bibitem{LL} Li CW, Liu XQ. 1997. Algebraic structure of multiple stochastic integrals
with respect to Brownian motions and Poisson processes.
\textit{Stochastics and Stoch.\ Reports} \textbf{61}, pp.~107--120.

\bibitem{L} Lyons T. 1998.
Differential equations driven by rough signals.
\textit{Rev. Mat. Iberoamericana}
\textbf{14}(2), pp.~215--310.

\bibitem{Magnus} Magnus W. 1954.
On the exponential solution of differential
equations for a linear operator.
\textit{Comm. Pure Appl. Math.} \textbf{7}, pp.~649--673.   

\bibitem{MW} Malham SJA, Wiese A. 2008.
Stochastic Lie group integrators,
\textit{SIAM J. Sci. Comput.} \textbf{30}(2), pp.~597--617

\bibitem{MW:Hopf} Malham SJA, Wiese A. 2009.
Stochastic expansions and Hopf algebras.
\textit{Proc. R. Soc. A} \textbf{465}, pp.~3729--3749.

\bibitem{Marcus} Marcus SI. 1978. Modeling and analysis of stochastic differential
equations driven by point processes.
\textit{IEEE Transactions on Information Theory} \textbf{IT-24}(2), pp.~164--172.

\bibitem{NPT} Novelli JC, Patras F, Thibon JY. 2011.
Natural endomorphisms of quasi-shuffle Hopf algebras. to appear
in \textit{Bull. Soc. Math. de France}.

\bibitem{O} Oksendal, B. 2003. \emph{Stochastic differential equations:
An introduction with applications}. Sixth edition, Springer.

\bibitem{B-LP} Platen E, Bruti--Liberati N. 2010.
\textit{Numerical solution of Stochastic differential equations
with jumps in finance}. Springer.

\bibitem{Protter} Protter PE. 2005.
\emph{Stochastic Integration and Differential Equations}.
Second Edition, Springer.

\bibitem{R} Reutenauer C. 1993. \emph{Free Lie algebras}.
London Mathematical Society Monographs New Series 7,
Oxford Science Publications.

\bibitem{Schutzenberger} Sch\"utzenberger MP. 1958/9. 
Sur une propi\'et\'e combinatoire des alg\'ebres de Lie libres 
pouvant \^etre utilis\'ee dans une probl\`eme
math\'ematiques appliqu\'ees. Seminar Dubriel--Jacotin Pisot
(Alg\'ebre et th\'eorie des nombres), Paris.

\bibitem{S} Strichartz RS. 1987.
The Campbell--Baker--Hausdorff--Dynkin formula
and solutions of differential equations,
\textit{J. Funct. Anal.\ }
\textbf{72}, pp.~320--345

\end{thebibliography}
\end{document}